\documentclass[12pt]{article}
\usepackage{amsmath,amssymb,amsthm,amsfonts,amsbsy, enumerate,graphicx,graphics,epstopdf}
\usepackage[dvips]{epsfig}
\usepackage{verbatim}

\theoremstyle{plain}
\newtheorem{theorem}{Theorem}[section]

\newtheorem{conj}[theorem]{Conjecture}
\newtheorem{lemma}[theorem]{Lemma}
\newtheorem{proposition}[theorem]{Proposition}

\begin{document}

\title{Edge-Disjoint Spanning Trees and Eigenvalues of Regular Graphs}
\author{Sebastian M. Cioab\u{a}\footnote{Department of Mathematical Sciences, University of Delaware, Newark, DE 19716-2553, {\tt cioaba@math.udel.edu}.
This work was partially supported by a grant from the Simons Foundation ($\#209309$ to Sebastian M. Cioab\u{a}).} \, and Wiseley Wong\footnote {Department of Mathematical Sciences, University of Delaware, Newark, DE 19716-2553, {\tt wwong@math.udel.edu}.}}
\date{March 12, 2012}
\maketitle

\begin{abstract}
Partially answering a question of Paul Seymour, we obtain a sufficient eigenvalue condition for the existence of $k$ edge-disjoint spanning trees in a regular graph, when $k\in \{2,3\}$. More precisely, we show that if the second largest eigenvalue of a $d$-regular graph $G$ is less than $d-\frac{2k-1}{d+1}$, then $G$ contains at least $k$ edge-disjoint spanning trees, when $k\in \{2,3\}$. We construct examples of graphs that show our bounds are essentially best possible. We conjecture that the above statement is true for any $k<d/2$.
\end{abstract}

\section{Introduction}

Our graph notation is standard (see West \cite{West} for undefined terms). The adjacency matrix of a graph $G$ with $n$ vertices has its rows and columns indexed after the vertices of $G$ and the  $(u,v)$-entry of $A$ is $1$ if $uv=\{u,v\}$ is an edge of $G$
and $0$ otherwise. If $G$ is undirected, then $A$ is symmetric. Therefore, its eigenvalues
are real numbers, and we order them as $\lambda_1\geq
\lambda_2\geq \dots \geq \lambda_n$. The Laplacian matrix $L$ of $G$ equals $D-A$, where $D$ is the diagonal degree matrix of $G$. The Laplacian matrix is positive semidefinite and we order its eigenvalues as $0=\mu_1\leq \mu_2\leq \dots \leq \mu_n$. It is well known that if $G$ is connected and $d$-regular, then $\mu_i=d-\lambda_i$ for each $1\leq i\leq n$, $\lambda_1=d$ and $\lambda_i<d$ for any $i\neq 1$ (see \cite{BH,GR}).

Kirchhoff Matrix Tree Theorem \cite{Kir1847} (see \cite[Section 1.3.5]{BH} or \cite[Section 13.2]{GR} for short proofs) is one of the classical results of combinatorics. It states that the number of spanning trees of a graph $G$ with $n$ vertices is the principal minor of the Laplacian matrix $L$ of the graph and consequently, equals $\frac{\prod_{i=2}^{n}\mu_i}{n}$. In particular, if $G$ is a $d$-regular graph, then the number of spanning trees of $G$ is $\frac{\prod_{i=2}^{n}(d-\lambda_i)}{n}$.

Motivated by these facts and by a question of Seymour \cite{Seymour}, in this paper, we find relations between the maximum number of edge-disjoint spanning trees (also called the spanning tree packing number or tree packing number; see Palmer \cite{Pa} for a survey of this parameter) and the eigenvalues of a regular graph. Let $\sigma(G)$ denote the maximum number of edge-disjoint spanning trees of $G$. Obviously, $G$ is connected if and only if $\sigma(G)\geq 1$. 

A classical result, due to Nash-Williams \cite{NW} and independently, Tutte \cite{Tutte} (see \cite{K} for a recent short constructive proof), states that a graph $G$ contains $k$ edge-disjoint spanning trees if and only if for any partition of its vertex set $V(G)=X_1\cup \dots \cup X_t$ into $t$ non-empty subsets, the following condition is satisfied:
\begin{equation}\label{NWT}
\sum_{1\leq i<j\leq t}e(X_i,X_j)\geq k(t-1)
\end{equation}

A simple consequence of Nash-Williams/Tutte Theorem is that if $G$ is a $2k$-edge-connected graph, then $\sigma(G)\geq k$ (see Kundu \cite{Ku}). Catlin \cite{Ca} (see also \cite{CLS}) improved this result and showed that a graph $G$ is $2k$-edge-connected if and only if the graph obtained from removing any $k$ edges from $G$ contains at least $k$ edge-disjoint spanning trees.

An obvious attempt to find relations between $\sigma(G)$ and the eigenvalues of $G$ is by using the relations between eigenvalues and edge-connectivity of a regular graph as well as the previous observations relating the edge-connectivity to $\sigma(G)$. Cioab\u{a} \cite{Ci} has proven that if $G$ is a $d$-regular graph and $2\leq r\leq d$ is an integer such that $\lambda_2<d-\frac{2(r-1)}{d+1}$, then $G$ is $r$-edge-connected. While not mentioned in \cite{Ci}, it can be shown that the upper bound above is essentially best possible. An obvious consequence of these facts is that if $G$ is a $d$-regular graph with $\lambda_2<d-\frac{2(2k-1)}{d+1}$ for some integer $k$, $2\leq k\leq \lfloor \frac{d}{2} \rfloor$, then $G$ is $2k$-edge-connected and consequently, $G$ contains $k$-edge-disjoint spanning trees. 

In this paper, we improve the bound above as follows. 
\begin{theorem}\label{sigma2}
If $d\geq 4$ is an integer and $G$ is a $d$-regular graph such that $\lambda_2(G)<d-\frac{3}{d+1}$, then $G$ contains at least $2$ edge-disjoint spanning trees.
\end{theorem}
We remark that the existence of $2$ edge-disjoint spanning trees in a graph implies some good properties (cf. \cite{OY}); for example, every graph $G$ with $\sigma(G)\geq 2$ has a cycle double cover (see \cite{OY} for more details). The proof of Theorem \ref{sigma2} is contained in Section \ref{2tree:sec}. In Section \ref{2tree:sec}, we will also show that Theorem \ref{sigma2} is essentially best possible by constructing examples of $d$-regular graphs $\mathcal{G}_d$ such that $\sigma(\mathcal{G}_d)<2$ and $\lambda_2(\mathcal{G}_d)\in \left(d-\frac{3}{d+2},d-\frac{3}{d+3}\right)$. In Section \ref{2tree:sec}, we will answer a question of Palmer \cite[Section 3.7, page 19]{Pa} by proving that the minimum number of vertices of a $d$-regular graph with edge-connectivity $2$ and spanning tree number $1$ is $3(d+1)$.
\begin{theorem}\label{sigma3}
If $d\geq 6$ is an integer and $G$ is a $d$-regular graph such that $\lambda_2(G)<d-\frac{5}{d+1}$, then $G$ contains at least $3$ edge-disjoint spanning trees.
\end{theorem}
The proof of this result is contained in Section \ref{3tree:sec}. In Section \ref{3tree:sec}, we will also show that Theorem \ref{sigma3} is essentially best possible by constructing examples of $d$-regular graphs $\mathcal{H}_d$ such that $\sigma(\mathcal{H}_d)<3$ and $\lambda_2(\mathcal{H}_d)\in \left[d-\frac{5}{d+1},d-\frac{5}{d+3}\right)$. We conclude the paper with some final remarks and open problems.

The main tools in our paper are Nash-Williams/Tutte Theorem stated above and eigenvalue interlacing described below (see also \cite{BH,GR,Hae,HJ}).
\begin{theorem}
Let $\lambda_j(M)$ be the $j$-th largest eigenvalue of a matrix $M$.  If $A$ is a real symmetric $n \times n$ matrix and $B$ is a
principal submatrix of $A$ with order $m \times m$, then for $1 \leq  i \leq  m$, 
\begin{equation}
\lambda_i(A) \geq \lambda_i(B) \geq \lambda_{n-m+i}(A).
\end{equation}
\end{theorem}
This theorem implies that if $H$ is an induced subgraph of a graph $G$, then the eigenvalues of $H$ interlace the eigenvalues of $G$. 

If $S$ and $T$ are disjoint subsets of the vertex set of $G$, then we denote by $E(S,T)$ the set of edges with one endpoint in $S$ and another endpoint in $T$. Also, let $e(S,T)=|E(S,T)|$. If $S$ is a subset of vertices of $G$, let $G[S]$ denote the subgraph of $G$ induced by $S$. The previous interlacing result implies that if $A$ and $B$ are two disjoint subsets of a graph $G$ such that $e(A,B)=0$, then the eigenvalues of $G[A\cup B]$ interlace the eigenvalues of $G$. As the spectrum of $G[A\cup B]$ is the union of the spectrum of $G[A]$ and the spectrum of $G[B]$ (this follows from $e(A,B)=0$), it follows that 
\begin{equation}\label{disjoint}
\lambda_2(G)\geq \lambda_2(G[A\cup B])\geq \min(\lambda_1(G[A]),\lambda_1(G[B]))\geq \min(\overline{d}(A),\overline{d}(B)),
\end{equation}
where $\overline{d}(S)$ denotes the average degree of $G[S]$.

Consider a partition $V(G)=V_1\cup \dots V_s$ of the vertex set of $G$ into $s$ non-empty subsets. For $1 \leq i, j \leq s$, let $b_{i,j}$ denote the average number of neighbors in $V_j$ of the vertices in $V_i$. The quotient matrix of this partition is the $s \times s$ matrix whose $(i,j)$-th entry equals $b_{i,j}$. A theorem of Haemers (see \cite{Hae} and also, \cite{BH,GR}) states that the eigenvalues of the quotient matrix interlace the eigenvalues of $G$. The previous partition is called equitable if for each $1\leq i,j\leq s$, any vertex $v\in V_i$ has exactly $b_{i,j}$ neighbors in $V_j$. In this case, the eigenvalues of the quotient matrix are eigenvalues of $G$ and the spectral radius of the quotient matrix equals the spectral radius of $G$ (see \cite{BH,GR,Hae} for more details).

\section{Eigenvalue condition for $2$ edge-disjoint spanning trees}\label{2tree:sec}

In this section, we give a proof of Theorem \ref{sigma2} showing that if $G$ is a $d$-regular graph such that $\lambda_2(G)<d-\frac{3}{d+1}$, then $G$ contains at least $2$ edge-disjoint spanning trees. We show that the bound $d-\frac{3}{d+1}$ is essentially best possible by constructing examples of $d$-regular graphs $\mathcal{G}_d$ having $\sigma(\mathcal{G}_d)<2$ and $d-\frac{3}{d+2}<\lambda_2(\mathcal{G}_d)<d-\frac{3}{d+3}$.

\begin{proof}[Proof of Theorem \ref{sigma2}]

We prove the contrapositive. Assume that $G$ does not contain $2$-edge-disjoint spanning trees. We will show that $\lambda_2(G)\geq d-\frac{3}{d+1}$.

By Nash-Williams/Tutte Theorem, there exists a partition of the vertex set of $G$ into $t$ subsets $X_1,\dots, X_t$ such that 
\begin{equation}
\sum_{1\leq i<j\leq t}e(X_i,X_j)\leq 2(t-1)-1=2t-3.
\end{equation}
It follows that 
\begin{equation}\label{degreesum2}
\sum_{i=1}^{t}r_i\leq 4t-6
\end{equation}
where $r_i=e(X_i,V\setminus X_i)$. 

Let $n_i=|X_i|$ for $1\leq i\leq t$. It is easy to see that $r_i\leq d-1$ implies $n_i\geq d+1$ for each $1\leq i\leq 3$.

If $t=2$, then $e(X_1,V\setminus X_1)=1$. By results of \cite{Ci}, it follows that $\lambda_2(G)>d-\frac{2}{d+4}>d-\frac{3}{d+1}$ and this finishes the proof of this case. Actually, we may assume $r_i\geq 2$ for every $1\leq i\leq t$ since $r_i=1$ and results of \cite{Ci} would imply $\lambda_2(G)>d-\frac{2}{d+4}>d-\frac{3}{d+1}$.

If $t=3$, then $r_1+r_2+r_3\leq 6$ which implies $r_1=r_2=r_3=2$. The only way this can happen is if $e(X_i,X_j)=1$ for every $1\leq i<j\leq 3$. Consider the partition of $G$ into $X_1, X_2$ and $X_3$.  The quotient matrix of this partition is
$$
A_3=
\begin{bmatrix}
d-\frac{2}{n_1} & \frac{1}{n_1} & \frac{1}{n_1}\\
\frac{1}{n_2} & d-\frac{2}{n_2} & \frac{1}{n_2}\\
\frac{1}{n_3} & \frac{1}{n_3} & d-\frac{2}{n_3}
\end{bmatrix}.
$$

The largest eigenvalue of $A_3$ is $d$ and the second eigenvalue of $A_3$ equals $$d-\frac{1}{n_1}-\frac{1}{n_2}-\frac{1}{n_3}+\sqrt{\frac{1}{n_1^2}+\frac{1}{n_2^2}+\frac{1}{n_3^2}-\frac{1}{n_1n_2}-\frac{1}{n_2n_3}-\frac{1}{n_3n_1}},$$
which is greater than $d-\frac{1}{n_1}-\frac{1}{n_2}-\frac{1}{n_3}$. Thus, eigenvalue interlacing and $n_i\geq d+1$ for $1\leq i\leq 3$ imply $\lambda_2(G)\geq \lambda_2(A_3)\geq d-\frac{3}{d+1}$. This finishes the proof of the case $t=3$.

Assume $t\geq 4$ from now on. Let $a$ denote the number of $r_i$'s that equal $2$ and $b$ denote the number of $r_j$'s that equal $3$. Using equation \eqref{degreesum2}, we get
\begin{equation*}
4t-6\geq \sum_{i=1}^{t}r_i\geq 2a+3b+4(t-a-b)=4t-2a-b,
\end{equation*}
which implies $2a+b\geq 6$.

Recall that $\overline{d}(A)$ denotes the average degree of the subgraph of $G$ induced by the subset $A\subset V(G)$.

If $a=0$, then $b\geq 6$. This implies that there exist two indices $1\leq i<j\leq t$ such that $r_i=r_j=3$ and $e(X_i,X_j)=0$. Eigenvalue interlacing \eqref{disjoint} implies $\lambda_2(G)\geq \lambda_2(G[X_i\cup X_j])\geq \min (\lambda_1(G[X_i]),\lambda_1(G[X_j]))\geq \min(\overline{d}(X_i),\overline{d}(X_j)\geq \min (d-\frac{3}{n_i},d-\frac{3}{n_j})\geq d-\frac{3}{d+1}$.

If $a=1$, then $b\geq 4$. This implies there exist two indices $1\leq i<j\leq t$ such that $r_i=2$, $r_j=3$ and $e(X_i,X_j)=0$. Eigenvalue interlacing \eqref{disjoint} implies $\lambda_2(G)\geq \lambda_2(G[X_i\cup X_j])\geq \min (\lambda_1(G[X_i]),\lambda_1(G[X_j]))\geq \min(\overline{d}(X_i),\overline{d}(X_j))\geq \min (d-\frac{2}{n_i},d-\frac{3}{n_j})\geq d-\frac{3}{d+1}$. 

If $a=2$, then $b\geq 2$. If there exist two indices $1\leq i<j\leq t$ such that $r_i=r_j=2$ and $e(X_i,X_j)=0$, then eigenvalue interlacing \eqref{disjoint} implies $\lambda_2(G)\geq \lambda_2(G[(X_i\cup X_j])\geq \min (\lambda_1(G[X_i]),\lambda_1(G[X_j]))\geq \min(\overline{d}(X_i),\overline{d}(X_j))\geq  \min (d-\frac{2}{n_i},d-\frac{2}{n_j})\geq d-\frac{2}{d+1}>d-\frac{3}{d+1}$. Otherwise, there exist two indices $1\leq p<q\leq t$ such that $r_p=2$, $r_q=3$ and $e(X_p,X_q)=0$. By a similar eigenvalue interlacing argument, we get $\lambda_2(G)\geq d-\frac{3}{d+1}$ in this case as well.

If $a=3$, then if there exist two indices $1\leq i<j\leq t$ such that $r_i=r_j=2$ and $e(X_i,X_j)=0$, then as before, eigenvalue interlacing \eqref{disjoint} implies $\lambda_2(G)\geq d-\frac{2}{d+1}>d-\frac{3}{d+1}$. This finishes the proof of Theorem \ref{sigma2}.
\end{proof}

%
%
\begin{figure}[h]
\begin{center}
\includegraphics[scale=0.5]{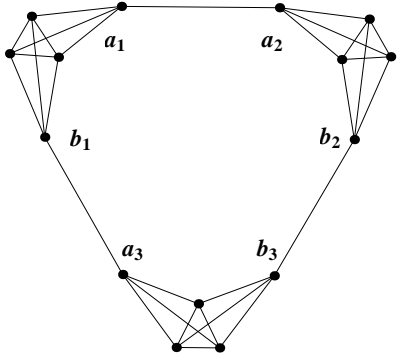}
\caption{The $4$-regular graph $\mathcal{G}_4$ with $\sigma(\mathcal{G}_4)=1$ and $3.5=4-\frac{3}{4+2}<\lambda_2(\mathcal{G}_4)\approx 3.569 < 4-\frac{3}{4+3}\approx 3.571$}
\end{center}
\end{figure}

We show that our bound is essentially best possible by presenting a family of $d$-regular graphs $\mathcal{G}_d$ with $d-\frac{3}{d+2}<\lambda_2(\mathcal{G}_d)<d-\frac{3}{d+3}$ and $\sigma(\mathcal{G}_d)=1$, for every $d\geq 4$.

For $d\geq 4$, consider three vertex disjoint copies $G_1,G_2,G_3$ of $K_{d+1}$ minus one edge. Let $a_i$ and $b_i$ be the two non adjacent vertices in $G_i$ for $1\leq i\leq 3$. Let $\mathcal{G}_d$ be the $d$-regular graph obtained by joining $a_1$ with $a_2$, $b_2$ and $b_3$ and $a_3$ and $b_1$.
The graph $\mathcal{G}_d$ has $3(d+1)$ vertices and is $d$-regular. The partition of the vertex set of $\mathcal{G}_d$ into $V(G_1), V(G_2), V(G_3)$ has the property that the number of edges between the parts equals $3$. By Nash-Williams/Tutte Theorem, this implies $\sigma(\mathcal{G}_d)<2$.

For $d\geq 4$, denote by $\theta_d$ the largest root of the cubic polynomial 
\begin{equation}
P_3(x)=x^3+(2-d)x^2+(1-2d)x+2d-3.
\end{equation}
\begin{lemma}
For every integer $d\geq 4$, the second largest eigenvalue of $\mathcal{G}_d$ is $\theta_d$.
\end{lemma}
\begin{proof}
Consider the following partition of the vertex set of $\mathcal{G}_d$ into nine parts: $V(G_1)\setminus \{a_1,b_1\}, V(G_2)\setminus \{a_2,b_2\},V(G_3)\setminus \{a_3,b_3\}, \{a_1\}, \{b_1\}, \{a_2\},\{b_2\}, \{a_3\} ,\{b_3\}$. This is an equitable partition whose quotient matrix is the following
\begin{equation}
A_9=\begin{bmatrix}
 d-2 & 0 & 0 & 1 & 1 & 0 & 0 & 0 & 0 \\
 0 & d-2 & 0 & 0 & 0 & 1 & 1 & 0 & 0 \\
 0 & 0 & d-2 & 0 & 0 & 0 & 0 & 1 & 1 \\
 d-1 & 0 & 0 & 0 & 0 & 1 & 0 & 0 & 0 \\
 d-1 & 0 & 0 & 0 & 0 & 0 & 0 & 1 & 0 \\
 0 & d-1 & 0 & 1 & 0 & 0 & 0 & 0 & 0 \\
 0 & d-1 & 0 & 0 & 0 & 0 & 0 & 0 & 1 \\
 0 & 0 & d-1 & 0 & 1 & 0 & 0 & 0 & 0 \\
 0 & 0 & d-1 & 0 & 0 & 0 & 1 & 0 & 0 \\
 \end{bmatrix}.
\end{equation}

The characteristic polynomial of $A_9$ is 
\begin{equation}
P_9(x)=(x-d)(x+1)^{2}[x^3+(2-d)x^2+(1-2d)x+2d-3]^{2}.
\end{equation}

Let $\lambda_2\geq \lambda_3\geq \lambda_4$ denote the solutions of the equation $x^3+(2-d)x^2+(1-2d)x+2d-3=0$. Because the above partition is equitable, it follows that $d, \lambda_2,\lambda_3,\lambda_4$ and $-1$ are eigenvalues of $\mathcal{G}_d$, and the multiplicity of each of them as an eigenvalue of $\mathcal{G}_d$ is at least $2$. 

We claim the spectrum of $\mathcal{G}_d$ is  
\begin{equation}\label{spectrumGd}
d^{(1)}, \lambda_2^{(2)}, \lambda_3^{(2)}, \lambda_4^{(2)}, (-1)^{(3d-4)}.
\end{equation}
It suffices to obtain $3d-4$ linearly independent eigenvectors corresponding to $-1$. Consider two distinct vertices $u_{1}$ and $u_{2}$ in $V(G_1)\setminus \{a_1,b_1\}$. Define an eigenvector where the entry corresponding to $u_{1}$ is $1$, the entry corresponding to $u_{2}$ is $-1$, and all the other entries are $0$.  We create $d-2$ eigenvectors by letting $u_{2}$ to be each of the $d-2$ vertices in $V(G_1)\setminus \{a_1,b_1,u_{1}\}$.  This can also be done to two vertices $u_{1}',u_{2}' \in V(G_2)\setminus \{a_2,b_2\}$ or two vertices $u_{1}'',u_{2}'' \in V(G_3)\setminus \{a_3,b_3\}$. This way, we obtain a total of $3d-6$ linearly independent eigenvectors corresponding to $-1$. Furthermore, define an vector with entries at three fixed vertices $u_{1}\in V(G_1)\setminus \{a_1,b_1\}, u_{1}' \in V(G_2)\setminus \{a_2,b_2\}, u_{1}''\in V(G_3)\setminus \{a_3,b_3\}$ equal to $-1$, with entries at $a_{1}, b_{2}, a_{3}$ equal to  $1$ and with entries $0$ everywhere else. It is easy to check this is an eigenvector corresponding to $0$.  To obtain the final eigenvector, define a new vector by setting the entries at three fixed vertices $u_{1}\in V(G_1)\setminus \{a_1,b_1\}, u_{1}' \in V(G_2)\setminus \{a_2,b_2\}, u_{1}''\in V(G_3)\setminus \{a_3,b_3\}$ to be $-1$, the entries at  $b_{1},a_{2},$ and $b_{3}$ to be $1$ and the remaining entries to be $0$. It is easy to check all these $3d-4$ vectors are linearly independent eigenvectors corresponding to eigenvalue $-1$. Having obtained the entire spectrum of $\mathcal{G}_{d}$, the second largest eigenvalue of $\mathcal{G}_{d}$ must be $\theta_d$.
\end{proof}

\begin{lemma}
For every integer $d\geq 4$, 
$$
d-\frac{3}{d+2}<\theta_d<d-\frac{3}{d+3}.
$$
\end{lemma}
\begin{proof}
We find that for $d\ge 4$,
$$P_{3}\left(d-\frac{3}{d+2}\right) = -\frac{3 \left(9+d \left(-2+d+d^2\right)\right)}{(2+d)^3}<0,$$
$$P_{3}\left(d-\frac{3}{d+3}\right)=\frac{-81+6 d^2}{(3+d)^3}>0,$$
and $P_{3}'(x)>0$ beyond $x=\frac{1}{3} (-1+2 d)<d-\frac{3}{d+3}$.  Hence,
\begin{equation}\label{thetad}
d-\frac{3}{d+2}<\theta_d<d-\frac{3}{d+3}
\end{equation}
for every $d\geq 4$.
\end{proof}

Palmer \cite{Pa} asked whether or not the graph $\mathcal{G}_4$ has the smallest number of vertices among all $4$-regular graphs with edge-connectivity $2$ and spanning tree number $1$. We answer this question affirmatively below.
\begin{proposition}
Let $d\geq 4$ be an integer. If $G$ is a $d$-regular graph such that $\kappa'(G)=2$ and $\sigma(G)=1$, then $G$ has at least $3(d+1)$ vertices. The only graph with these properties and $3(d+1)$ vertices is $\mathcal{G}_d$.
\end{proposition}
\begin{proof}
As $\sigma(G)=1<2$, by Nash-Williams/Tutte theorem, there exists a partition $V(G)=X_1\cup \dots \cup X_t$ such that $e(X_1,\dots,X_t)\leq 2t-3$. This implies $r_1+\dots +r_t\leq 4t-6$. As $\kappa'(G)=2$, it means that $r_i\geq 2$ for each $1\leq i\leq t$ which implies $4t-6\geq 2t$ and thus, $t\geq 3$. 

If $t=3$, then $r_i=2$ for each $1\leq i\leq 3$ and thus, $e(X_i,X_j)=1$ for each $1\leq i\neq j\leq 3$. As $d\geq 4$ and $r_i=2$, we deduce that $|X_i|\geq d+1$. Equality happens if and only if $X_i$ induces a $K_{d+1}$ without one edge. Thus, we obtain that $|V(G)|=|X_1|+|X_2|+|X_3|\geq 3(d+1)$ with equality if and only if $G=\mathcal{G}_d$.

If $t\geq 4$, then let $\alpha$ denote the number of $X_i$'s such that $|X_i|\geq d+1$. If $\alpha\geq 3$, then $|V(G)|>3(d+1)$ and we are done. Otherwise, $\alpha\leq 2$. Note that if $|X_i|\leq d$, then $r_i\geq d$. Thus,
$$
4t-6\geq r_1+\dots +r_t\geq 2\alpha+d(t-\alpha)=dt-(d-2)\alpha
$$
which implies $(d-2)\alpha\geq (d-4)t+6$. As $\alpha \leq 2$ and $t\geq 4$, we obtain $2(d-2)\geq (d-4)4+6$ which is equivalent to $2d\leq 6$, contradiction. This finishes our proof.
\end{proof}

\section{Eigenvalue condition for $3$ edge-disjoint spanning trees}\label{3tree:sec}

In this section, we give a proof of Theorem \ref{sigma3} showing that if $G$ is a $d$-regular graph such that $\lambda_2(G)<d-\frac{5}{d+1}$, then $G$ contains at least $3$ edge-disjoint spanning trees. We show that the bound $d-\frac{5}{d+1}$ is essentially best possible by constructing examples of $d$-regular graphs $\mathcal{H}_d$ having $\sigma(\mathcal{H}_d)<3$ and $d-\frac{5}{d+1}\le \lambda_2(\mathcal{H}_d)<d-\frac{5}{d+3}$.

\begin{proof}[Proof of Theorem \ref{sigma3}]
We prove the contrapositive. We assume that $G$ does not contain $3$-edge-disjoint spanning trees and we prove that $\lambda_2(G)\geq d-\frac{5}{d+1}$. 

By Nash-Williams/Tutte Theorem, there exists a partition of the vertex set of $G$ into $t$ subsets $X_1,\dots, X_t$ such that 
$$
\sum_{1\leq i<j\leq t}e(X_i,X_j)\leq 3(t-1)-1=3t-4.
$$
It follows that $\sum_{i=1}^{t}r_i\leq 6t-8$, where $r_i=e(X_i,V\setminus X_i)$.

If $r_i\leq 2$ for some $i$ between $1$ and $t$, then by results of \cite{Ci}, it follows that $\lambda_2(G)\geq d-\frac{4}{d+3}>d-\frac{5}{d+1}$. 

Assume $r_i\geq 3$ for each $1\leq i \leq t$ from now on. Let $a=|\{i:1\leq i\leq t, r_i=3\}|, b=|\{i:1\leq i\leq t, r_i=4\}|$ and $c=|\{i:1\leq i \leq t, r_i=5\}|$. We get that
$$
6t-8\geq r_1+\dots +r_t\geq 3a+4b+5c+6(t-a-b-c)
$$
which implies
\begin{equation}\label{abc1}
3a+2b+c\geq 8.
\end{equation}

If for some $1\leq i<j\leq t$, we have $e(X_i,X_j)=0$ and $\max(r_i,r_j)\leq 5$, then eigenvalue interlacing \eqref{disjoint} implies $\lambda_2(G)\geq \lambda_2(G[X_i\cup X_j])\geq \min (\lambda_2(G[X_i]),\lambda_2(G[X_j]))\geq \min(\overline{d}(X_i),\overline{d}(X_j))\geq d-\frac{5}{d+1}$ and we would be done. Thus, we may assume that 
\begin{equation}\label{nonempty}
e(X_i,X_j)\geq 1
\end{equation}
for every $1\leq i<j\leq t$ with $\max(r_i,r_j)\leq 5$. Similar arguments imply for example that 
\begin{equation}\label{abc2}
a+b+c\leq 6, a+b\leq 5, a\leq 4.
\end{equation}

For the rest of the proof, we have to consider the following cases:

{\bf Case 1.} $a\geq 2$.

The inequality $\sum_{1\leq i<j\leq t}e(X_i,X_j)\leq 3t-4$ implies $t\geq 3$.

As $a=|\{i:r_i=3\}|$, assume without loss of generality that $r_1=r_2=3$. 
Because $G$ is connected, this implies $e(X_1,X_2)<3$. Otherwise, $e(X_1\cup X_2,V(G)\setminus (X_1\cup X_2))=0$, contradiction.

If $e(X_1,X_2)=2$, then $e(X_1\cup X_2,V(G)\setminus (X_1\cup X_2))=2$. Using the results in \cite{Ci}, this implies $\lambda_2(G)\geq d-\frac{4}{d+2}>d-\frac{5}{d+1}$ and finishes the proof. 


Thus, we may assume $e(X_1,X_2)=1$. Let $Y_3=V(G)\setminus (X_1\cup X_2)$. As $r_1=r_2=3$, we deduce that $e(X_1,Y_3)=e(X_2,Y_3)=2$. This means $e(Y_3,V(G)\setminus Y_3)=4$ and since $d\geq 6$, this implies $n'_3:=|Y_3|\geq d+1$. 

Consider the partition of the vertex set of $G$ into three parts: $X_1, X_2$ and $Y_3$. The quotient matrix of this partition is 
\begin{equation*}
B_3=
\begin{bmatrix}
d-\frac{3}{n_1} & \frac{1}{n_1} & \frac{2}{n_1}\\
\frac{1}{n_2} & d-\frac{3}{n_2} & \frac{2}{n_2}\\
\frac{2}{n'_3}&\frac{2}{n'_3}& d-\frac{4}{n'_3}
\end{bmatrix}.
\end{equation*}


The largest eigenvalue of $B_3$ is $d$.  Eigenvalue interlacing and $n_1,n_2, n_3'\geq d+1$  imply 
\begin{align*}
\lambda_2(G)&\geq \lambda_2(B_3)\geq \frac{tr(B_3)-d}{2}\geq d-\frac{3}{2 n_1}-\frac{3}{2 n_2}-\frac{2}{n_3'}\\
&\geq d-\frac{3}{2(d+1)}-\frac{3}{2(d+1)}-\frac{2}{d+1}=d-\frac{5}{d+1}.
\end{align*}

This finishes the proof of this case.

{\bf Case 2.} $a=1$.

Inequalities \eqref{abc1} and \eqref{abc2} imply $2b+c\geq 5\geq b+c$. Actually, because we assumed that $e(X_i,X_j)\geq 1$ for every $1\leq i\neq j \leq t$ with $\max(r_i,r_j)\leq 5$, we deduce that $b+c\leq 3$. Otherwise, if $b+c\geq 4$, then there exists $i\neq j$ such that $r_i=3, r_j\in \{4,5\}$ and $e(X_i,X_j)=0$.

The only solution of the previous inequalities is $b=2$ and $c=1$. Without loss of generality, we may assume $r_1=3, r_2=r_3=4$ and $r_4=5$. Using the facts of the previous paragraph, we deduce that $e(X_1,X_j)=1$ for each $2\leq j\leq 4$ and $e(X_i,X_j)\geq 1$ for each $2\leq i\neq j\leq 4$.

If $e(X_2,X_3)\geq 3$, then $e(X_2,X_4)=0$ which is a contradiction with the first paragraph of this subcase.

If $e(X_2,X_3)=2$, then $t\geq 5$ and $e(X_1\cup X_2\cup X_3\cup X_4, V(G)\setminus (X_1\cup X_2\cup X_3\cup X_4))=2$. Using results from \cite{Ci}, it follows that $\lambda_2(G)\geq d-\frac{4}{d+2}>d-\frac{5}{d+1}$ which finishes the proof of this subcase.

If $e(X_2,X_3)=1$, then there are some subcases to consider:

\begin{enumerate}

\item If $e(X_2,X_4)=e(X_3,X_4)=1$, then $t\geq 5$. If $Y_5:=V(G)\setminus (X_1\cup X_2\cup X_3\cup X_4)$, then $e(X_4,Y_5)=2, e(X_3,Y_5)=e(X_2,Y_5)=1$. These facts imply $e(Y_5,V(G)\setminus Y_5)=4$ and $e(X_1,Y_5)=0$. As $d\geq 6$, it follows that $n'_5:=|Y_5|\geq d+1$. Eigenvalue interlacing \eqref{disjoint} implies 
\begin{align*}
\lambda_2(G)&\geq \lambda_2(G[X_1\cup Y_5])\geq \min(\lambda_1(G[X_1]),\lambda_1(G[Y_5]))\geq \min(\overline{d}(X_1),\overline{d}(Y_5))\\
&\geq \min\left(d-\frac{3}{n_1},d-\frac{4}{n'_5}\right)\geq d-\frac{4}{d+1}>d-\frac{5}{d+1}
\end{align*}
which finishes the proof of this subcase.

\item If $e(X_2,X_4)=2$ and $e(X_3,X_4)=1$, then $t\geq 5$. If $Y_5:=V(G)\setminus (X_1\cup X_2\cup X_3\cup X_4)$, then $e(X_4,Y_5)=e(X_3,Y_5)=1$.  These facts imply $e(Y_5,V(G)\setminus Y_5)=2$. Using results in \cite{Ci}, we obtain $\lambda_2(G)>d-\frac{4}{d+2}>d-\frac{5}{d+1}$ which finishes the proof of this subcase.

\item If $e(X_2,X_4)=1$ and $e(X_3,X_4)=2$, then the proof is similar to the previous case and we omit the details.

\item If $e(X_2,X_4)=e(X_3,X_4)=2$, then $t=4$. Consider the partition of the vertex set of $G$ into three parts: $X_1, X_2, X_3\cup X_4$. The quotient matrix of this partition is
\begin{equation*}
C_3=\begin{bmatrix}
d-\frac{3}{n_1} & \frac{1}{n_1} & \frac{2}{n_1}\\
\frac{1}{n_2} & d-\frac{4}{n_2} & \frac{3}{n_2} \\
\frac{3}{n'_3} & \frac{2}{n'_3} & d-\frac{5}{n'_3}
\end{bmatrix}
\end{equation*}
where $n'_3=|X_3\cup X_4|=|X_3|+|X_4|\geq 2(d+1)$.



The largest eigenvalue of $C_3$ is $d$.  Eigenvalue interlacing and $n_1,n_2\ge d+1$, $n_3'\ge 2(d+1)$  imply 
\begin{align*}
\lambda_2(G)&\geq \lambda_2(C_3)\geq \frac{tr(C_3)-d}{2}\geq d-\frac{3}{2 n_1}-\frac{2}{ n_2}-\frac{5}{2n_3'} \\
&\geq d-\frac{3}{2(d+1)}-\frac{2}{d+1}-\frac{5}{4(d+1)}\ge d-\frac{4.75}{d+1}> d-\frac{5}{d+1}.
\end{align*}

\end{enumerate}

{\bf Case 3.} $a=0$.

Inequalities \eqref{abc1} and \eqref{abc2} imply $2b+c\geq 8, b+c\leq 6, b\leq 5$. 

If $b=0$, then $c\geq 8$ and $c\leq 6$ which is a contradiction that finishes the proof of this subcase.


If $b=1$, then $c\geq 6$ and $c\leq 5$ which is a contradiction that finishes the proof of our subcase.


If $b=2$, then $c\geq 4$ which implies that there exists $i\ne j$ such that $e(X_{i},X_{j})=0$ and $r_{i}=4$ and $r_{j}\in \{4,5\}$. This contradicts  \eqref{nonempty} and finishes the proof. 



If $b=3$, then $c\geq 2$. Assume that $c=2$ first. Without loss of generality, assume $r_{1}=r_{2}=r_{3}=4$ and $r_{4}=r_{5}=5$.  \eqref{nonempty} implies that $e(X_i,X_j)=1$ for each $1\leq i<j\leq 5$ except when $i=4$ and $j=5$ where $e(X_4,X_5)=2$.

Consider the partition of the vertex set of $G$ into three parts: $X_{1}, X_{2}\cup X_{3}$, and $X_{4}\cup X_{5}$.  The quotient matrix of this partition is
$$D_{3}=\begin{bmatrix}
d-\frac{4}{n_{1}}&& \frac{2}{n_{1}} &&\frac{2}{n_{1}}\\
\frac{2}{n_{2}'}&&d-\frac{6}{n_{2}'}&&\frac{4}{n_{2}'}\\
\frac{2}{n_{3}'}&&\frac{4}{n_{3}'}&&d-\frac{6}{n_{3}'}
\end{bmatrix}$$
where $n'_2=|X_2\cup X_3|=|X_2|+|X_3|\geq 2(d+1)$ and $n'_3=|X_4\cup X_5|=|X_4|+|X_5|\geq 2(d+1)$.

The largest eigenvalue of $D_3$ is $d$.  Eigenvalue interlacing and $n_1\ge d+1$, $n_2',n_3'\ge 2(d+1)$  imply 
\begin{align*}
\lambda_2(G)&\geq \lambda_2(D_3)\geq \frac{tr(D_3)-d}{2}\geq d-\frac{2}{n_1}-\frac{3}{n_2'}-\frac{3}{n_3'}\\
&\geq d-\frac{2}{d+1}-\frac{3}{2(d+1)}-\frac{3}{2(d+1)}=d-\frac{5}{d+1}.
\end{align*}

This finishes the proof of this subcase.

If $c\geq 3$, then since $b=3$, it follows that there exists $i\ne j$ such that $e(X_{i},X_{j})=0$ and $r_{i}=4$ and $r_{j}\in \{4,5\}$. This contradicts \eqref{nonempty} and finish the proof of this subcase.


If $b=4$, we have inequality \eqref{abc2} implies $c\leq 2$. If $c=2$, then there exist $i\neq j$ such that $e(X_i,X_j)=0, r_i=4$ and $r_j\in \{4,5\}$. This contradicts \eqref{nonempty} and finishes the proof of this subcase.

Suppose $c=0$.  Without loss of generality, assume that $r_i=4$ for $1\leq i\leq 4$. If $t=4$, then \eqref{nonempty} implies that the graph $G$ is necessarily of the form shown in Figure  \ref{b=4c=0}.

\begin{figure}[h]
\centering
\includegraphics[scale=0.45]{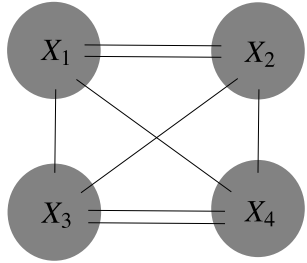}
\caption{The structure of $G$ when $a=0$, $b=4$, $c=0$, and $t=4$.}
\label{b=4c=0}
\end{figure}

Consider the partition of the vertex set of $G$ into three parts: $X_{1}, X_{2},X_{3}\cup X_{4}$.  The quotient matrix of this partition is
$$E_{3}=\begin{bmatrix}
d-\frac{4}{n_{1}}&& \frac{2}{n_{1}} &&\frac{2}{n_{1}}\\
\frac{2}{n_{2}}&&d-\frac{4}{n_{2}}&&\frac{2}{n_{2}}\\
\frac{2}{n_{3}'}&&\frac{2}{n_{3}'}&&d-\frac{4}{n_{3}'}
\end{bmatrix}$$
where $n'_3=|X_3\cup X_4|=|X_3|+|X_4|\geq 2(d+1)$.

The largest eigenvalue of $E_3$ is $d$.
Eigenvalue interlacing and $n_1,n_2 \ge d+1$, $n_3'\ge 2(d+1)$  imply 
\begin{align*}
\lambda_2(G)&\geq \lambda_2(E_3)\geq \frac{tr(E_3)-d}{2}\geq d-\frac{2}{n_1}-\frac{2}{n_2}-\frac{2}{n_3'}\\
& \geq d-\frac{2}{d+1}-\frac{2}{d+1}-\frac{2}{2(d+1)}= d-\frac{5}{d+1}.
\end{align*}


If $t\geq 5$, then there are two possibilities: either $e(X_i,X_j)=1$ for each $1\leq i<j\leq 4$ or without loss of generality, $e(X_i,X_j)=1$ for each $1\leq i<j\leq 4$ except for $i=1$ and $j=2$ where $e(X_1,X_2)=2$.

In the first situation, if $Y_5:=V(G)\setminus (X_1\cup X_2\cup X_3\cup X_4)$, then $e(X_i,Y_5)=1$ for each $1\leq i\leq 4$ and thus, $e(Y_5, V(G)\setminus Y_5)=4$. This implies $|Y_5|\geq d+1$. Consider the partition of $V(G)$ into three parts $X_1, X_2\cup X_3, X_4\cup Y_5$. The quotient matrix of this partition is 
$$
F_3=
\begin{bmatrix}
d-\frac{4}{n_1} & \frac{2}{n_1} & \frac{2}{n_1}\\
\frac{2}{n'_2} & d-\frac{6}{n'_2} & \frac{4}{n'_2}\\
\frac{2}{n'_3} & \frac{4}{n'_3} & d-\frac{6}{n'_3}
\end{bmatrix}
$$
where $n'_2=|X_2\cup X_3|=|X_2|+|X_3|\geq 2(d+1)$ and $n'_3=|X_4\cup Y_5|=|X_4|+|Y_5|\geq 2(d+1)$.

The largest eigenvalue of $F_3$ is $d$.
Eigenvalue interlacing and $n_1 \ge d+1$, $n_2',n_3'\ge 2(d+1)$ imply 
\begin{align*}
\lambda_2(G)&\geq \lambda_2(F_3)\geq \frac{tr(F_3)-d}{2}\geq d-\frac{2}{n_1}-\frac{3}{n_2'}-\frac{3}{n_3'}\\
&\geq d-\frac{2}{d+1}-\frac{3}{2(d+1)}-\frac{3}{2(d+1)}=d-\frac{5}{d+1},
\end{align*}
which finishes the proof of this subcase.

\begin{figure}[h]
\centering
\includegraphics[scale=0.7]{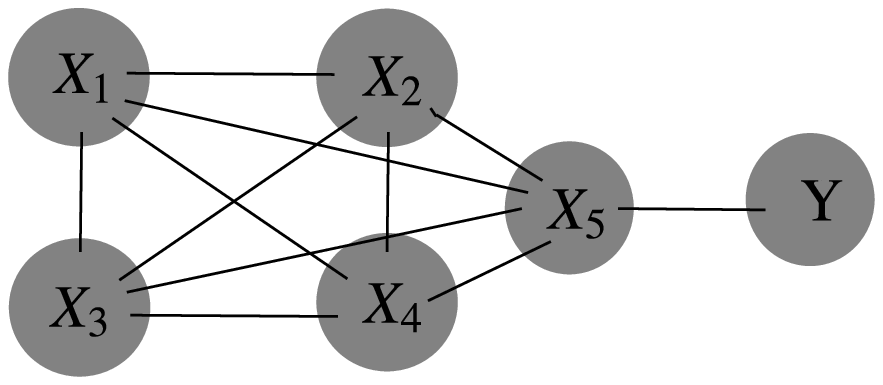}
\caption{The structure of $G$ when $a=0$, $b=4$, $c=1$, and $t\ge 5$.}
\label{b=4c=1}
\end{figure}

In the second situation, if $Y_5:=V(G)\setminus (X_1\cup X_2\cup X_3\cup X_4)$ then $e(X_1,Y_5)=e(X_2,Y_5)=0$ and $e(X_3,Y_5)=e(X_4,Y_5)=1$. This implies $e(Y_5,V(G)\setminus Y_5)=2$. By results of \cite{Ci}, we deduce that $\lambda_2(G)\geq d-\frac{4}{d+2}>d-\frac{5}{d+1}$ which finishes the proof of this subcase.

Assume that $c=1$. Without loss of generality, assume that $r_{i}=4$ for $1\le i\le 4$, and $r_{5}=5$. Our assumption \eqref{nonempty} implies that the graph is necessarily of the form shown in Figure \ref{b=4c=1},  where $Y$ is a component that necessarily joins to $X_{5}$.  By results of \cite{Ci}, it follows that $\lambda_2(G)>d-\frac{2}{d+4}>d-\frac{5}{d+1}$ and this finishes the proof of this case.

If $b=5$,  then $c=0$ by  \eqref{nonempty}. Also, by \eqref{nonempty}, it follows that $t=5$ and $e(X_i,X_j)=1$ for each $1\leq i<j\leq 5$. Consider the partition of the vertex set of $G$ into three parts: $X_{1}, X_{2}\cup X_{3}, X_{4}\cup X_{5}$.  The quotient matrix of this partition is
$$G_{3}=\begin{bmatrix}
d-\frac{4}{n_{1}}&& \frac{2}{n_{1}} &&\frac{2}{n_{1}}\\
\frac{2}{n_{2}'}&&d-\frac{6}{n_{2}'}&&\frac{4}{n_{2}'}\\
\frac{2}{n_{3}'}&&\frac{4}{n_{3}'}&&d-\frac{6}{n_{3}'}
\end{bmatrix},$$
which is identical to the quotient matrix $F_3$ in a previous case, which yields $\lambda_2(G)\ge   d-\frac{5}{d+1}.$


If $b>5$, then \eqref{nonempty} will yield a contradiction. This finishes the proof of Theorem \ref{sigma3}.
\end{proof}



We show that our bound is essentially best possible by presenting a family of $d$-regular graphs $\mathcal{H}_d$ with $d-\frac{5}{d+1}\le \lambda_2(\mathcal{H}_d)<d-\frac{5}{d+3}$ and $\sigma(\mathcal{H}_d)=2$, for every $d\geq 6$.

For $d\geq 6$, consider the graph obtained from $K_{d+1}$ by removing two disjoint edges. Consider now 5 vertex disjoint copies $H_1,H_2,H_3, H_{4}, H_{5}$ of this graph. For each copy $H_i, 1\leq i\leq 5$, denote the two pairs of non-adjacent vertices in $H_i$ by $a_i, c_i$ and $b_i, d_i$. Let $\mathcal{H}_d$ be the $d$-regular graph whose vertex set is $\cup_{i=1}^{5}V(H_i)$ and whose edge set is the union $\cup_{i=1}^{5}E(H_i)$ with the following set of $10$ edges:
$$\{b_{1}a_{2},b_{2}a_{3},b_{3}a_{4},b_{4}a_{5},b_5a_1, c_{1}d_{3},c_{3}d_{5},c_{5}d_{2},c_{2}d_{4},c_{4}d_{1}\}.$$

\begin{figure}[h]
\centering
\includegraphics[scale=0.4]{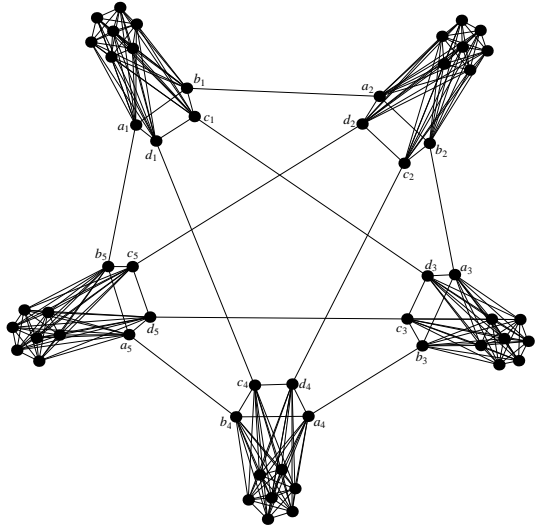}
\caption{The 10-regular graph $\mathcal{H}_{10}$ with $\sigma\left(\mathcal{H}_{10}\right)=2$ and $9.545 \approx 10-\frac{5}{10+1} <\lambda_2\left({H}_{10}\right)\approx 9.609 <10-\frac{5}{10+3}\approx 9.615.$ }
\end{figure}

The graph $\mathcal{H}_d$ is $d$-regular and has $5(d+1)$ vertices. The partition of the vertex set of $\mathcal{H}_d$ into the five parts: $V(H_1), V(H_2), V(H_3), V(H_{4}), V(H_{5})$ has the property that the number of edges between the parts equals $10<12=3(5-1)$. By Nash-Williams/Tutte Theorem, this implies $\sigma(\mathcal{H}_d)<3$.

For $d\geq 6$, denote by $\gamma_d$ the largest root of the polynomial 
\begin{align*}
&x^{10}+(8-2 d) x^9+(d^2-16 d+30) x^8+(8 d^2-50 d+58) x^7+(20 d^2-66 d+36) x^6+\\
&(8 d^2+18 d-70) x^5+(-29 d^2+140 d-146) x^4+(-20 d^2+57 d-21) x^3+(14 d^2-83 d+109) x^{2}+\\
&(4 d^2-13 d+5) x-d^2+5 d-5.
\end{align*}

\begin{lemma}
For every integer $d\geq 6$, the second largest eigenvalue of $\mathcal{H}_d$ is $\gamma_d$.
\end{lemma}
\begin{proof}
Consider the following partition of the vertex set of $\mathcal{H}_d$ into 25 parts: 5 parts of the form $V(H_i)\setminus \{a_i,b_i,c_{i},d_{i}\},$ $i=1,2,3,4,5$.  The remaining 20 parts consist of the 20 individual vertices $\{a_{i}\}, \{b_{i}\}, \{c_{i}\},\{d_{i}\}$, $i=1,2,3,4,5$.  This partition is equitable and the characteristic polynomial of its quotient matrix (which is described in Section \ref{25quotient}) is
\begin{align*}
P_{25}(x)&=(x-d)(x-1)(x+1)^{2}(x+3)[x^{10}+(8-2 d) x^9+(d^2-16 d+30) x^8+\\
&(8 d^2-50 d+58) x^7+(20 d^2-66 d+36) x^6+(8 d^2+18 d-70) x^5+\\
&(-29 d^2+140 d-146) x^4+(-20 d^2+57 d-21) x^3+(14 d^2-83 d+109) x^2\\
&+(4 d^2-13 d+5) x-d^2+5 d-5]^{2}.
\end{align*}

Let $\lambda_2\geq \lambda_3\geq...\ge \lambda_{11}$ denote the solutions of the degree 10 polynomial $P_{10}(x)$. Because the partition is equitable, it follows that these 10 solutions, $d, 1,-1,$ and $-3$ are eigenvalues of $\mathcal{H}_d$, including multiplicity. 

We claim the spectrum of $\mathcal{H}_d$ is  
\begin{equation}\label{spectrumHd}
d^{(1)}, 1^{(1)},-3^{(1)},-1^{(5d-18)}, \lambda_i^{(2)} \hspace{.2in} \text{for } i=2,3,...,11. 
\end{equation}
It suffices to obtain $5d-18$ linearly independent eigenvectors corresponding to $-1$. Consider two distinct vertices $u_{1}^{1}$ and $u_{2}^{1}$ in $V(H_1)\setminus \{a_1,b_1,c_{1},d_{1}\}$. Define a vector where the entry corresponding to $u_{1}^{1}$ is 1, the entry corresponding to $u_{2}^{1}$ is $-1$, and all other entries are $0$. This is an eigenvector corresponding to the eigenvalue $-1$. We can create $d-4$ eigenvectors by letting $u_{2}^{1}$ to be each of the $d-4$ vertices in $V(H_1)\setminus \{a_1,b_1,c_{1},d_{1},u_{1}^1\}$. This can also be applied to $2$ vertices $u_{1}^{i},u_{2}^{i}$ in $V(H_i)\setminus \{a_i,b_i,c_{i},d_{i}\},$ for $i=2,3,4,5$. This way, we obtain a total of $5d-20$ linearly independent eigenvectors corresponding to the eigenvalue $-1$.

Furthermore, define a vector whose entry at some fixed vertex $u_{1}^{i}\in V(H_i)\setminus \{a_i,b_i,c_i,d_i\}$ is $-2$, whose entries at $a_{i}$ and $d_{i}$ are $1$, for each $1\leq i\leq 5$ and whose remaining entries are $0$. Define another vector whose entries at a fixed vertex $u_{1}^{i}\in V(H_i)\setminus \{a_i,b_i,c_i,d_i\}$ is $-2$, whose entries at $b_{i}$ and $c_{i}$ are $1$, for each $1\leq i\leq 5$ and whose remaining entries are $0$. These last two vectors are also eigenvectors corresponding to the eigenvalue $-1$. It is easy to check that all these $5d-18$ vectors we have constructed are linearly independent eigenvectors corresponding to the eigenvalue $-1$. By obtaining the entire spectrum of $\mathcal{H}_{d}$, we conclude that the second largest eigenvalue of $\mathcal{H}_{d}$ is $\gamma_d$.
\end{proof}

\begin{lemma}\label{gammad}
For every integer $d\geq 6$,
$$
d-\frac{5}{d+1}\le\gamma_d<d-\frac{5}{d+3}.
$$
\end{lemma}
\begin{proof}
The lower bound follows directly from Theorem \ref{sigma3} as $\sigma(\mathcal{H}_d)<3$. Moreover, by some technical calculations (done in Mathematica and included in Section \ref{25quotient})
$$P_{10}^{(n)}\left(d-\frac{5}{d+3}\right)>0, \text{for } n=0,1,...,10.$$
Descartes' Rule of Signs implies $\gamma_{d}<d-\frac{5}{d+3}$. Hence,
\begin{equation}\label{theta2d}
d-\frac{5}{d+1}\le\gamma_d<d-\frac{5}{d+3}
\end{equation}
for every $d\geq 6$.
\end{proof}




\section{Final Remarks}

In this paper, we studied the relations between the eigenvalues of a regular graph and its spanning tree packing number. Based on the results contained in this paper, we make the following conjecture. 
\begin{conj}
Let $d\geq 8$ and $4\leq k\leq \lfloor \frac{d}{2} \rfloor$ be two integers. If $G$ is a $d$-regular graph such that $\lambda_2(G)<d-\frac{2k-1}{d+1}$, then $G$ contains at least $k$ edge-disjoint spanning trees.
\end{conj}

Let $\omega(H)$ denote the number of components of the graph $H$. 
The vertex-toughness of $G$ is defined as $\min \frac{|S|}{\omega(G\setminus S)}$, where the minimum is taken over all subsets of vertices $S$ whose removal disconnects $G$. Alon \cite{Alon1} and independently, Brouwer \cite{Brouwer1} have found close relations between the eigenvalues of a regular graph and its vertex-toughness. These connections were used by Alon in \cite{Alon1} to disprove a conjecture of Chv\'{a}tal that a graph with sufficiently large vertex-toughness is pancyclic. For $c\geq 1$, the higher order edge-toughness $\tau_c(G)$ is defined as 
$$
\tau_c(G):=\min \frac{|X|}{\omega(G\setminus X)-c}
$$
where the minimum is taken over all subsets $X$ of edges of $G$ with the property $\omega(G\setminus X)>c$ (see Chen, Koh and Peng \cite{CKP} or Catlin, Lai and Shao \cite{CLS} for more details). The Nash-Williams/Tutte Theorem states that $\sigma(G)=\lfloor \tau_1(G)\rfloor$. Cunningham \cite{Cun} generalized this result and showed that if $\tau_1(G)\geq \frac{p}{q}$ for some natural numbers $p$ and $q$, then $G$ contains $p$ spanning trees (repetitions allowed) such that each edge of G lies in at most $q$ of the $p$ trees. Chen, Koh and Peng \cite{CKP} proved that $\tau_c(G)\geq k$ if and only if $G$ contains at least $c$ edge-disjoint forests with exactly $c$ components. It would be interesting to find connections between the eigenvalues of the adjacency matrix (or of the Laplacian) of a graph $G$ and $\tau_c(G)$. 

Another question of interest is to determine sufficient eigenvalue condition for the existence of {\em nice} spanning trees in pseudorandom graphs. A lot of work has been done on this problem in the case of random graphs (see Krivelevich \cite{Kri} for example).


\section*{Acknowledgments} We thank the referee for some useful remarks.

\vskip 5ex

\section{Calculations for Lemma \ref{gammad}}\label{25quotient}

\subsection{Justify characteristic polynomial in 25 parts}
The following is the characteristic polynomial of the equitable partition in 25 parts:

\tiny{$\text{Factor}[\text{CharacteristicPolynomial}[\left(
\begin{array}{ccccccccccccccccccccccccc}
 d-4 & 0 & 0 & 0 & 0 & 1 & 1 & 1 & 1 & 0 & 0 & 0 & 0 & 0 & 0 & 0 & 0 & 0 & 0 & 0 & 0 & 0 & 0 & 0 & 0 \\
 0 & d-4 & 0 & 0 & 0 & 0 & 0 & 0 & 0 & 1 & 1 & 1 & 1 & 0 & 0 & 0 & 0 & 0 & 0 & 0 & 0 & 0 & 0 & 0 & 0 \\
 0 & 0 & d-4 & 0 & 0 & 0 & 0 & 0 & 0 & 0 & 0 & 0 & 0 & 1 & 1 & 1 & 1 & 0 & 0 & 0 & 0 & 0 & 0 & 0 & 0 \\
 0 & 0 & 0 & d-4 & 0 & 0 & 0 & 0 & 0 & 0 & 0 & 0 & 0 & 0 & 0 & 0 & 0 & 1 & 1 & 1 & 1 & 0 & 0 & 0 & 0 \\
 0 & 0 & 0 & 0 & d-4 & 0 & 0 & 0 & 0 & 0 & 0 & 0 & 0 & 0 & 0 & 0 & 0 & 0 & 0 & 0 & 0 & 1 & 1 & 1 & 1 \\
 d-3 & 0 & 0 & 0 & 0 & 0 & 1 & 0 & 1 & 0 & 0 & 0 & 0 & 0 & 0 & 0 & 0 & 0 & 0 & 0 & 0 & 0 & 1 & 0 & 0 \\
 d-3 & 0 & 0 & 0 & 0 & 1 & 0 & 1 & 0 & 1 & 0 & 0 & 0 & 0 & 0 & 0 & 0 & 0 & 0 & 0 & 0 & 0 & 0 & 0 & 0 \\
 d-3 & 0 & 0 & 0 & 0 & 0 & 1 & 0 & 1 & 0 & 0 & 0 & 0 & 0 & 0 & 0 & 1 & 0 & 0 & 0 & 0 & 0 & 0 & 0 & 0 \\
 d-3 & 0 & 0 & 0 & 0 & 1 & 0 & 1 & 0 & 0 & 0 & 0 & 0 & 0 & 0 & 0 & 0 & 0 & 0 & 1 & 0 & 0 & 0 & 0 & 0 \\
 0 & d-3 & 0 & 0 & 0 & 0 & 1 & 0 & 0 & 0 & 1 & 0 & 1 & 0 & 0 & 0 & 0 & 0 & 0 & 0 & 0 & 0 & 0 & 0 & 0 \\
 0 & d-3 & 0 & 0 & 0 & 0 & 0 & 0 & 0 & 1 & 0 & 1 & 0 & 1 & 0 & 0 & 0 & 0 & 0 & 0 & 0 & 0 & 0 & 0 & 0 \\
 0 & d-3 & 0 & 0 & 0 & 0 & 0 & 0 & 0 & 0 & 1 & 0 & 1 & 0 & 0 & 0 & 0 & 0 & 0 & 0 & 1 & 0 & 0 & 0 & 0 \\
 0 & d-3 & 0 & 0 & 0 & 0 & 0 & 0 & 0 & 1 & 0 & 1 & 0 & 0 & 0 & 0 & 0 & 0 & 0 & 0 & 0 & 0 & 0 & 1 & 0 \\
 0 & 0 & d-3 & 0 & 0 & 0 & 0 & 0 & 0 & 0 & 1 & 0 & 0 & 0 & 1 & 0 & 1 & 0 & 0 & 0 & 0 & 0 & 0 & 0 & 0 \\
 0 & 0 & d-3 & 0 & 0 & 0 & 0 & 0 & 0 & 0 & 0 & 0 & 0 & 1 & 0 & 1 & 0 & 1 & 0 & 0 & 0 & 0 & 0 & 0 & 0 \\
 0 & 0 & d-3 & 0 & 0 & 0 & 0 & 0 & 0 & 0 & 0 & 0 & 0 & 0 & 1 & 0 & 1 & 0 & 0 & 0 & 0 & 0 & 0 & 0 & 1 \\
 0 & 0 & d-3 & 0 & 0 & 0 & 0 & 1 & 0 & 0 & 0 & 0 & 0 & 1 & 0 & 1 & 0 & 0 & 0 & 0 & 0 & 0 & 0 & 0 & 0 \\
 0 & 0 & 0 & d-3 & 0 & 0 & 0 & 0 & 0 & 0 & 0 & 0 & 0 & 0 & 1 & 0 & 0 & 0 & 1 & 0 & 1 & 0 & 0 & 0 & 0 \\
 0 & 0 & 0 & d-3 & 0 & 0 & 0 & 0 & 0 & 0 & 0 & 0 & 0 & 0 & 0 & 0 & 0 & 1 & 0 & 1 & 0 & 1 & 0 & 0 & 0 \\
 0 & 0 & 0 & d-3 & 0 & 0 & 0 & 0 & 1 & 0 & 0 & 0 & 0 & 0 & 0 & 0 & 0 & 0 & 1 & 0 & 1 & 0 & 0 & 0 & 0 \\
 0 & 0 & 0 & d-3 & 0 & 0 & 0 & 0 & 0 & 0 & 0 & 1 & 0 & 0 & 0 & 0 & 0 & 1 & 0 & 1 & 0 & 0 & 0 & 0 & 0 \\
 0 & 0 & 0 & 0 & d-3 & 0 & 0 & 0 & 0 & 0 & 0 & 0 & 0 & 0 & 0 & 0 & 0 & 0 & 1 & 0 & 0 & 0 & 1 & 0 & 1 \\
 0 & 0 & 0 & 0 & d-3 & 1 & 0 & 0 & 0 & 0 & 0 & 0 & 0 & 0 & 0 & 0 & 0 & 0 & 0 & 0 & 0 & 1 & 0 & 1 & 0 \\
 0 & 0 & 0 & 0 & d-3 & 0 & 0 & 0 & 0 & 0 & 0 & 0 & 1 & 0 & 0 & 0 & 0 & 0 & 0 & 0 & 0 & 0 & 1 & 0 & 1 \\
 0 & 0 & 0 & 0 & d-3 & 0 & 0 & 0 & 0 & 0 & 0 & 0 & 0 & 0 & 0 & 1 & 0 & 0 & 0 & 0 & 0 & 1 & 0 & 1 & 0
\end{array}
\right),x]]$}

\begin{align*}
&(d-x) (-1+x) (1+x)^2 (3+x) (-5+5 d-d^2+5 x-13 d x+4 d^2 x+109 x^2-83 d x^2+14 d^2 x^2-21 \\
&x^3+57 d x^3-20 d^2 x^3-146 x^4+140 d x^4-29 d^2 x^4-70 x^5+18 d x^5+8 d^2 x^5+36 x^6-66 \\ 
&x^6+20 d^2 x^6+58 x^7-50 d x^7+8 d^2 x^7+30 x^8-16 d x^8+d^2 x^8+8 x^9-2 d x^9+x^{10} )^2
\end{align*}







\subsection{Justify $P_{10}^{(n)}\left(d-\frac{5}{d+3}\right)>0, \text{for } n=0,1,...,10.$}

\subsubsection{n=0}
$\text{Factor}\left[
\begin{array}{c}
-5+5 d-d^2+5 x-13 d x+4 d^2 x+109 x^2-83 d x^2+14 d^2 x^2-21 x^3+\\
57 d x^3-20 d^2 x^3-146 x^4+140 d x^4-29 d^2 x^4-70 x^5+18 d x^5+\\
8 d^2 x^5+36 x^6-66 d x^6+20 d^2 x^6+58 x^7-50 d x^7+8 d^2 x^7+\\
30 x^8-16 d x^8+d^2 x^8+8 x^9-2 d x^9+x^{10}
\end{array}
\text{/.}x\to d-5/(d+3)\right]$

$\frac{5 \left(209081+2789848 d+4225996 d^2-7988400 d^3-2586890 d^4+3149694 d^5+1156227 d^6-317856 d^7-185275 d^8-9630 d^9+7239 d^{10}+1412 d^{11}+79 d^{12}\right)}{(3+d)^{10}}$

\bigskip

Looking at the numerator,

$$209081+2789848 d+4225996 d^2-7988400 d^3-2586890 d^4+3149694 d^5+1156227 d^6-317856 d^7-185275 d^8-9630 d^9+7239 d^{10}+1412 d^{11}+79 d^{12}$$
$$\geq  209081+2789848 d+4225996 d^2-7988400 d^3-2586890 d^4+3149694 \left(6^2\right)d^3+1156227\left(6^2\right) d^4-317856 d^7-185275 d^8-9630 d^9+7239 \left(6^3\right)d^7+1412 \left(6^3\right)d^8+79 \left(6^3\right)d^9$$
$$=209081+2789848 d+4225996 d^2+105400584 d^3+39037282 d^4+1245768 d^7+119717 d^8+7434 d^9>0.$$

\subsubsection{n=1}

$\text{Apart}\left[\text{FullSimplify}\left[D\left[
\begin{array}{c}
-5+5 d-d^2+5 x-13 d x+4 d^2 x+109 x^2-83 d x^2+14 d^2 x^2-21 x^3+57 d x^3-\\
20 d^2 x^3-146 x^4+140 d x^4-29 d^2 x^4-70 x^5+18 d x^5+8 d^2 x^5+36 x^6-\\
66 d x^6+20 d^2 x^6+58 x^7-50 d x^7+8 d^2 x^7+30 x^8-16 d x^8+d^2 x^8+\\
8 x^9-2 d x^9+x^{10}
\end{array}
,x\right]\text{/.}x\to d-5/(d+3)\right]\right]$

$-154125-6265 d+9235 d^2-1605 d^3-80 d^4+40 d^5-\frac{19531250}{(3+d)^9}-\frac{56250000}{(3+d)^8}-\frac{43125000}{(3+d)^7}+\frac{14000000}{(3+d)^6}+\frac{26231250}{(3+d)^5}+\frac{250000}{(3+d)^4}-\frac{6723000}{(3+d)^3}-\frac{224000}{(3+d)^2}+\frac{981525}{3+d}$
\bigskip

Looking at the fraction terms,

$\text{Together}\left[-\frac{19531250}{(3+d)^9}-\frac{56250000}{(3+d)^8}-\frac{43125000}{(3+d)^7}+\frac{14000000}{(3+d)^6}+\frac{26231250}{(3+d)^5}+\frac{250000}{(3+d)^4}-\frac{6723000}{(3+d)^3}-\frac{224000}{(3+d)^2}+\frac{981525}{3+d}\right]$

$\frac{25 \left(121436221+368991216 d+491609352 d^2+377696288 d^3+179037720 d^4+52838632 d^5+9436692 d^6+933304 d^7+39261 d^8\right)}{(3+d)^9}$

The expression is positive.  The only concern now are the terms $-154125-6265 d+9235 d^2-1605 d^3-80 d^4+40 d^5$.  Direct calculations for $d=6$ and $7$ yield the values 1425 and 184220, respectively.  For $d\ge8$,
$$(-154125-6265 d+9235 d^2)-1605 d^3-80 d^4+40 d^5 $$
$$=9235d+(d-1)(9235)d-6265d-154125+80d^{4}+(d-2)(40)d^4-80d^4-1605d^3$$
$$\ge 9235 d + (7) (9235) (8) - 6265 d - 154125+ 80d^4+(6)(40)(8)d^3-80d^{4}-1605d^{3}$$
$$=  (1920-1605)d^{3}+(9235-6265)d+(517160-154125)>0.$$

\subsubsection{n=2}

$\text{Apart}\left[\text{FullSimplify}\left[D\left[
\begin{array}{c}
-5+5 d-d^2+5 x-13 d x+4 d^2 x+109 x^2-83 d x^2+\\
14 d^2 x^2-21 x^3+57 d x^3-20 d^2 x^3-146 x^4+\\
140 d x^4-29 d^2 x^4-70 x^5+18 d x^5+8 d^2 x^5+\\
36 x^6-66 d x^6+20 d^2 x^6+58 x^7-50 d x^7+8 d^2 x^7+\\
30 x^8-16 d x^8+d^2 x^8+8 x^9-2 d x^9+x^{10}
\end{array}
,\{x,2\}\right]\text{/.}x\to d-5/(d+3)\right]\right]$

$-501172+218908 d-37582 d^2-2480 d^3+2472 d^4-344 d^5-60 d^6+16 d^7+2 d^8+\frac{35156250}{(3+d)^8}+\frac{90000000}{(3+d)^7}+\frac{54750000}{(3+d)^6}-\frac{30800000}{(3+d)^5}-\frac{34412500}{(3+d)^4}+\frac{4300000}{(3+d)^3}+\frac{8668800}{(3+d)^2}-\frac{574400}{3+d}$
\bigskip

Looking at the fraction terms and $2d^{8}$,
\bigskip

$\text{Together}\left[\frac{35156250}{(3+d)^8}+\frac{90000000}{(3+d)^7}+\frac{54750000}{(3+d)^6}-\frac{30800000}{(3+d)^5}-\frac{34412500}{(3+d)^4}+\frac{4300000}{(3+d)^3}+\frac{8668800}{(3+d)^2}-\frac{574400}{3+d}+2d^{8}\right]$

$\frac{1}{(3+d)^8}2 \left(\begin{array}{c}
1643568075+3659898600 d+3340851900 d^2+1497989000 d^3+328783750 d^4+25888400 d^5-\\
1696800 d^6-287200 d^7+6561 d^8+17496 d^9+20412 d^{10}+13608 d^{11}+5670 d^{12}+1512 d^{13}\\
+252 d^{14}+24 d^{15}+d^{16}\end{array}\right)$
\bigskip

By comparing terms, the expression is positive.  The only concern now are the terms $-501172+218908 d-37582 d^2-2480 d^3+2472 d^4-344 d^5-60 d^6+16 d^7$.   Direct calculations for $d=6$ and $7$ yield the values 1132028 and 4610438, respectively.  Clearly we have for the first two terms that $-501172+218908d>0$.   Now assume $d\ge 8$.  Looking at the next  three terms,
$$-37582 d^2-2480 d^3+2472 d^4= 4944d^3+(d-2)(2472)d^3-2480d^3-37583d^2$$
$$\geq  4944d^3+(6)(2472)(8)d^2-2480d^3-37583d^2>0$$

For the final three terms,
$$-344 d^5-60 d^6+16 d^7=64d^6+16(d-4)d^6-60d^6-344d^5\geq  64d^6+16(4)(8)d^5-60d^6-344d^5>0.$$

\subsubsection{n=3}

$\text{Apart}\left[\text{FullSimplify}\left[D\left[
\begin{array}{c}
-5+5 d-d^2+5 x-13 d x+4 d^2 x+109 x^2-83 d x^2+\\
14 d^2 x^2-21 x^3+57 d x^3-20 d^2 x^3-146 x^4+\\
140 d x^4-29 d^2 x^4-70 x^5+18 d x^5+8 d^2 x^5+\\
36 x^6-66 d x^6+20 d^2 x^6+58 x^7-50 d x^7+8 d^2 x^7+\\
30 x^8-16 d x^8+d^2 x^8+8 x^9-2 d x^9+x^{10}
\end{array}
,\{x,3\}\right]\text{/.}x\to d-5/(d+3)\right]\right]$

$2377554-293322 d-71280 d^2+40944 d^3-5340 d^4-1380 d^5+336 d^6+48 d^7-\frac{56250000}{(3+d)^7}-\frac{126000000}{(3+d)^6}-\frac{56700000}{(3+d)^5}+\frac{50400000}{(3+d)^4}+\frac{35947500}{(3+d)^3}-\frac{10020000}{(3+d)^2}-\frac{8360520}{3+d}$
\bigskip

Looking at the fraction terms and $48d^{7}$,
\bigskip

$\text{Together}\left[-\frac{56250000}{(3+d)^7}-\frac{126000000}{(3+d)^6}-\frac{56700000}{(3+d)^5}+\frac{50400000}{(3+d)^4}+\frac{35947500}{(3+d)^3}-\frac{10020000}{(3+d)^2}-\frac{8360520}{3+d}+48d^7\right]$

$\frac{1}{(3+d)^7}12 \left(\begin{array}{c}
-433473465-955900680 d-877113900 d^2-411225900 d^3-103585225 d^4-\\
13375780 d^5-696710 d^6+8748 d^7+20412 d^8+20412 d^9+11340 d^{10}+\\
3780 d^{11}+756 d^{12}+84 d^{13}+4 d^{14}\end{array}\right)$
\bigskip

Looking at the numerator,

$$-433473465-955900680 d-877113900 d^2-411225900 d^3-103585225 d^4-13375780 d^5-696710 d^6+8748 d^7+20412 d^8+20412 d^9+11340 d^{10}+3780 d^{11}+756 d^{12}+84 d^{13}+4 d^{14}$$
$$\ge -433473465-955900680 d-877113900 d^2-411225900 d^3-103585225 d^4-13375780 d^5-696710 d^6+8748 d^7+20412\left(6^8\right)$$
$$+20412 \left(6^8\right)d+11340 \left(6^8\right)d^2+3780 \left(6^8\right)d^3+756\left(6^8\right) d^4+84 \left(6^8\right)d^5+4 \left(6^8\right)d^6$$
$$=33850848327+33328421112 d+18169731540 d^2+5937722580 d^3+1166204471 d^4+127711964 d^5+6021754 d^6+8748 d^7>0.$$

The only concern now are the terms $2377554-293322 d-71280 d^2+40944 d^3-5340 d^4-1380 d^5+336 d^6$.    Direct calculations for $d=6$ and $7$ yield the values 4920342 and 14390436, respectively.  We ignore the first positive constant, and assume $d\ge 8$.  Looking at the next three terms,
$$-293322 d - 71280 d^2 + 40944 d^3 =  81888 d^2 + (d - 2) (40944) d^2 - 71280 d^2 - 293322 d $$
$$\ge 81888 d^2 + (6) (40944) (8) d - 71280 d^2 - 293322 d>0$$

For the final three terms,
$$-5340 d^4 - 1380 d^5 + 336 d^6 = 1680 d^5 + (d - 5) 336 d^5 - 1380 d^5 - 5340 d^4$$ 
$$\ge1680 d^5 + (3) 336 (8) d^4 - 1380 d^5 - 5340 d^4 > 0.$$

\subsubsection{n=4}

$\text{Apart}\left[\text{FullSimplify}\left[D\left[
\begin{array}{c}-5+5 d-d^2+5 x-13 d x+4 d^2 x+109 x^2-83 d x^2+14 d^2 x^2-21 x^3+\\
57 d x^3-20 d^2 x^3-146 x^4+140 d x^4-29 d^2 x^4-70 x^5+18 d x^5+8 d^2 x^5+\\
36 x^6-66 d x^6+20 d^2 x^6+58 x^7-50 d x^7+8 d^2 x^7+30 x^8-16 d x^8+d^2 x^8+\\
8 x^9-2 d x^9+x^{10}\end{array},\{x,4\}\right]\text{/.}x\to d-5/(d+3)\right]\right]$

$-285504-1017840 d+396024 d^2-41280 d^3-18000 d^4+4032 d^5+672 d^6+\frac{78750000}{(3+d)^6}+\frac{151200000}{(3+d)^5}+\frac{44100000}{(3+d)^4}-\frac{63840000}{(3+d)^3}-\frac{28026000}{(3+d)^2}+\frac{13488000}{3+d}$
\bigskip

Looking at the fraction terms and $672d^6$,
\bigskip

$\text{Together}\left[672 d^6+\frac{78750000}{(3+d)^6}+\frac{151200000}{(3+d)^5}+\frac{44100000}{(3+d)^4}-\frac{63840000}{(3+d)^3}-\frac{28026000}{(3+d)^2}+\frac{13488000}{3+d}\right]$

$\frac{48 \left(4438500+23499000 d+33289500 d^2+16953500 d^3+3631125 d^4+281000 d^5+10206 d^6+20412 d^7+17010 d^8+7560 d^9+1890 d^{10}+252 d^{11}+14 d^{12}\right)}{(3+d)^6}$

This expression is positive.  The only concern now are the terms $-285504-1017840 d+396024 d^2-41280 d^3-18000 d^4+4032 d^5$.  Direct calculations for $d=6$ and 7 yield the values 6972672 and 22383576, respectively.  Now assume $d\ge 8$.  Looking at the first 3 terms,
$$-285504-1017840 d+396024 d^2=1188072d+(d-3)396024 d-1017840 d-285504\geq  1188072d+(5)396024 (8)-1017840 d-285504>0.$$

For the final three terms,
$$-41280 d^3-18000 d^4+4032 d^5=20160d^4+(d-5)4032 d^4-18000 d^4-41280 d^3\geq  20160d^4+(3)4032 (8)d^3-18000 d^4-41280 d^3>0.$$

\subsubsection{n=5}

$\text{Apart}\left[\text{FullSimplify}\left[D\left[
\begin{array}{c}-5+5 d-d^2+5 x-13 d x+4 d^2 x+109 x^2-83 d x^2+14 d^2 x^2-21 x^3+\\
57 d x^3-20 d^2 x^3-146 x^4+140 d x^4-29 d^2 x^4-70 x^5+18 d x^5+8 d^2 x^5+\\
36 x^6-66 d x^6+20 d^2 x^6+58 x^7-50 d x^7+8 d^2 x^7+30 x^8-16 d x^8+d^2 x^8+\\
8 x^9-2 d x^9+x^{10}\end{array},\{x,5\}\right]\text{/.}x\to d-5/(d+3)\right]\right]$

$-8576400+2476080 d-152400 d^2-162000 d^3+33600 d^4+6720 d^5-\frac{94500000}{(3+d)^5}-\frac{151200000}{(3+d)^4}-\frac{20160000}{(3+d)^3}+\frac{61824000}{(3+d)^2}+\frac{14024400}{3+d}$

\bigskip

Looking at the fraction terms,
\bigskip

$\text{Together}\left[-\frac{94500000}{(3+d)^5}-\frac{151200000}{(3+d)^4}-\frac{20160000}{(3+d)^3}+\frac{61824000}{(3+d)^2}+\frac{14024400}{3+d}\right]$

$\frac{1200 \left(1729737+2426436 d+1077978 d^2+191764 d^3+11687 d^4\right)}{(3+d)^5}$

This expression is positive.  The only concern now are the terms $-8576400+2476080 d-152400 d^2-162000 d^3+33600 d^4+6720 d^5$.  Clearly for the first two terms we have $-8576400+2476080 d>0$ for $d\ge 6$.  Looking at the four remaining terms,

$$-152400 d^2-162000 d^3+33600 d^4+6720 d^5\geq  -152400 d^2-162000 d^3+33600(36) d^2+6720(36) d^3= -152400 d^2-162000 d^3+1209600 d^2+241920 d^3>0.$$

\subsubsection{n=6}

$\text{Apart}\left[\text{FullSimplify}\left[D\left[
\begin{array}{c}-5+5 d-d^2+5 x-13 d x+4 d^2 x+109 x^2-83 d x^2+14 d^2 x^2-21 x^3+\\
57 d x^3-20 d^2 x^3-146 x^4+140 d x^4-29 d^2 x^4-70 x^5+18 d x^5+8 d^2 x^5+\\
36 x^6-66 d x^6+20 d^2 x^6+58 x^7-50 d x^7+8 d^2 x^7+30 x^8-16 d x^8+d^2 x^8+\\
8 x^9-2 d x^9+x^{10}\end{array},\{x,6\}\right]\text{/.}x\to d-5/(d+3)\right]\right]$

$9349920+244800 d-1044000 d^2+201600 d^3+50400 d^4+\frac{94500000}{(3+d)^4}+\frac{120960000}{(3+d)^3}-\frac{3024000}{(3+d)^2}-\frac{43545600}{3+d}$

\bigskip

Looking at the fraction terms and $50400 d^4+9349920+244800 d$,
\bigskip

$\text{Together}\left[\frac{94500000}{(3+d)^4}+\frac{120960000}{(3+d)^3}-\frac{3024000}{(3+d)^2}-\frac{43545600}{3+d}+50400 d^4+9349920+244800 d\right]$

$\frac{1440 \left(8178-30066 d+94722 d^2+56856 d^3+11368 d^4+3950 d^5+1890 d^6+420 d^7+35 d^8\right)}{(3+d)^4}$

This expression is clearly positive for $d\ge 6$.  The only terms left are $-1044000 d^2+201600 d^3$, and we get
$$-1044000 d^2+201600 d^3\ge -1044000 d^2+201600(6) d^2\ge-1044000d^2+1209600d^2 > 0.$$

\subsubsection{n=7}

$\text{Apart}\left[\text{FullSimplify}\left[D\left[
\begin{array}{c}-5+5 d-d^2+5 x-13 d x+4 d^2 x+109 x^2-83 d x^2+14 d^2 x^2-21 x^3+\\
57 d x^3-20 d^2 x^3-146 x^4+140 d x^4-29 d^2 x^4-70 x^5+18 d x^5+8 d^2 x^5+\\
36 x^6-66 d x^6+20 d^2 x^6+58 x^7-50 d x^7+8 d^2 x^7+30 x^8-16 d x^8+d^2 x^8+\\
8 x^9-2 d x^9+x^{10}\end{array},\{x,7\}\right]\text{/.}x\to d-5/(d+3)\right]\right]$

$5937120-4687200 d+846720 d^2+282240 d^3-\frac{75600000}{(3+d)^3}-\frac{72576000}{(3+d)^2}+\frac{13305600}{3+d}$

\bigskip

Looking at the fraction terms and $8282240 d^3$,
\bigskip

$\text{Together}\left[8282240 d^3-\frac{75600000}{(3+d)^3}-\frac{72576000}{(3+d)^2}+\frac{13305600}{3+d}\right]$

$\frac{640 \left(-271215+11340 d+20790 d^2+349407 d^3+349407 d^4+116469 d^5+12941 d^6\right)}{(3+d)^3}$

This expression is clearly positive for $d\ge 6$.  The only remaining terms are $5937120-4687200 d+846720 d^2$.  We have
$$5937120-4687200 d+846720 d^2\ge 5937120-4687200 d+846720(6) d =5937120-4687200 d+ 5080320d>0.$$

\subsubsection{n=8}

$\text{Apart}\left[\text{FullSimplify}\left[D\left[
\begin{array}{c}-5+5 d-d^2+5 x-13 d x+4 d^2 x+109 x^2-83 d x^2+14 d^2 x^2-21 x^3+\\
57 d x^3-20 d^2 x^3-146 x^4+140 d x^4-29 d^2 x^4-70 x^5+18 d x^5+8 d^2 x^5+\\
36 x^6-66 d x^6+20 d^2 x^6+58 x^7-50 d x^7+8 d^2 x^7+30 x^8-16 d x^8+d^2 x^8+\\
8 x^9-2 d x^9+x^{10}\end{array},\{x,8\}\right]\text{/.}x\to d-5/(d+3)\right]\right]$

$-13305600+2257920 d+1128960 d^2+\frac{45360000}{(3+d)^2}+\frac{29030400}{3+d}$

At $d=6$, the value is 44670080.  Clearly the expression is increasing for $d\ge 6$, and hence always positive for $d\ge 6$.

\subsubsection{n=9}

$\text{Apart}\left[\text{FullSimplify}\left[D\left[
\begin{array}{c}-5+5 d-d^2+5 x-13 d x+4 d^2 x+109 x^2-83 d x^2+14 d^2 x^2-21 x^3+\\
57 d x^3-20 d^2 x^3-146 x^4+140 d x^4-29 d^2 x^4-70 x^5+18 d x^5+8 d^2 x^5+\\
36 x^6-66 d x^6+20 d^2 x^6+58 x^7-50 d x^7+8 d^2 x^7+30 x^8-16 d x^8+d^2 x^8+\\
8 x^9-2 d x^9+x^{10}\end{array},\{x,9\}\right]\text{/.}x\to d-5/(d+3)\right]\right]$

$2903040+2903040 d-\frac{18144000}{3+d}$

At $d=6$, the value is 18305280.  Clearly the expression is increasing for $d\ge 6$, and hence always positive for $d\ge 6$.  

\subsubsection{n=10}
The value will be $10! > 0.$

\end{document}